\documentclass[runningheads]{llncs}

\usepackage[table]{xcolor}
\definecolor{lightgray}{gray}{0.9}
\usepackage{hyperref}

\usepackage{amssymb, amsmath}
\usepackage{mathrsfs}
\usepackage{cite}
\usepackage{enumitem}
\usepackage{cases}
\usepackage[font={it}, margin=1cm]{caption}
\usepackage{verbatim}
\usepackage{tikz}
\usepackage{tikz-cd} 
\usetikzlibrary{decorations.markings, decorations.pathreplacing, positioning,arrows, matrix, decorations.pathmorphing, shapes.geometric, calc}
\usetikzlibrary{backgrounds, decorations}
 
\newtheorem{thmletter}{Theorem}

\newcommand{\inj}{\operatorname{inj}}

\newcommand{\PCP}{\operatorname{PCP}}%_{FG}}
\newcommand{\PCPI}{\operatorname{PCP}^{\inj}}%_{FG}}
\newcommand{\PCPnoI}{\operatorname{PCP}^{(\neg\inj, \neg\inj)}}%_{FG}}
\newcommand{\PCPoneI}{\operatorname{PCP}^{(\neg\inj, \inj)}}%_{FG}}
\newcommand{\PCPtwoI}{\operatorname{PCP}^{(\inj, \inj)}}%_{FG}}
\newcommand{\GPCP}{\operatorname{GPCP}}%_{FG}}
\newcommand{\GPCPI}{\operatorname{GPCP}^{\inj}}%_{FG}}
\newcommand{\GPCPnoI}{\operatorname{GPCP}^{(\neg\inj, \neg\inj)}}%_{FG}}
\newcommand{\GPCPoneI}{\operatorname{GPCP}^{(\neg\inj, \inj)}}%_{FG}}
\newcommand{\GPCPtwoI}{\operatorname{GPCP}^{(\inj, \inj)}}%_{FG}}
\newcommand{\ELPCP}{\operatorname{PCP}_{\mathcal{EL}}}%_{FG}}
\newcommand{\ELPCPI}{\operatorname{PCP}_{\mathcal{EL}}^{\inj}}%_{FG}}
\newcommand{\ELPCPnoI}{\operatorname{PCP}_{\mathcal{EL}}^{(\neg\inj, \neg\inj)}}%_{FG}}
\newcommand{\ELPCPoneI}{\operatorname{PCP}_{\mathcal{EL}}^{(\neg\inj, \inj)}}%_{FG}}
\newcommand{\ELPCPtwoI}{\operatorname{PCP}_{\mathcal{EL}}^{(\inj, \inj)}}%_{FG}}
\newcommand{\CPCP}{\operatorname{PCP}_{\mathcal{R}}}%_{FG}}
\newcommand{\CPCPI}{\operatorname{PCP}_{\mathcal{R}}^{\inj}}%_{FG}}
\newcommand{\CPCPoneI}{\operatorname{PCP}_{\mathcal{R}}^{(\neg\inj, \inj)}}%_{FG}}
\newcommand{\CPCPtwoI}{\operatorname{PCP}_{\mathcal{R}}^{(\inj, \inj)}}%_{FG}}
\newcommand{\GPCPC}{\operatorname{GPCP}^{\operatorname{CI}}}%_{FG}}
\newcommand{\GPCPCI}{\operatorname{GPCP}^{\inj+\operatorname{CI}}}%_{FG}}
\newcommand{\GPCPConeI}{\operatorname{GPCP}^{(\neg\inj, \inj)+\operatorname{CI}}}%_{FG}}
\newcommand{\GPCPCtwoI}{\operatorname{GPCP}^{(\inj, \inj)+\operatorname{CI}}}%_{FG}}
\newcommand{\PCPC}{\operatorname{PCP}^{\operatorname{CI}}}%_{FG}}
\newcommand{\PCPCI}{\operatorname{PCP}^{\inj+\operatorname{CI}}}%_{FG}}
\newcommand{\AEP}{\operatorname{AEP}}
\newcommand{\BP}{\operatorname{BP}}%_{FG}}
\newcommand{\BPoneI}{\operatorname{BP}^{(\neg\inj, \inj)}}%_{FG}}
\newcommand{\BPtwoI}{\operatorname{BP}^{(\inj, \inj)}}%_{FG}}
\newcommand{\RPoneI}{\operatorname{RP}^{(\neg\inj, \inj)}}%_{FG}}
\newcommand{\RPtwoI}{\operatorname{RP}^{(\inj, \inj)}}%_{FG}}
\newcommand{\AEPoneIb}{\operatorname{BP}^{(\neg\inj, \inj)}}%_{FG}}
\newcommand{\AEPtwoIb}{\operatorname{BP}^{(\inj, \inj)}}%_{FG}}
\newcommand{\ERP}{\operatorname{RP}}%_{FG}}
\newcommand{\rk}{\operatorname{rk}}
\newcommand{\eq}{\operatorname{Eq}}
\newcommand{\im}{\operatorname{Im}}
\newcommand{\fix}{\operatorname{Fix}}

\begin{document}
\title{Variations on the Post Correspondence Problem for free groups\thanks{Supported by EPSRC grant EP/R035814/1}}
%\titlerunning{The $\PCP$ versus the $\GPCP$}
\author{Laura Ciobanu
%\inst{1}
\orcidID{0000-0002-9451-1471} \and
Alan D. Logan
%\inst{1}
\orcidID{0000-0003-1767-6798}}
\authorrunning{L. Ciobanu and A. D. Logan}
\institute{Heriot-Watt University,  Edinburgh EH14 4AS,
 Scotland\\
\email{\{L.Ciobanu,A.Logan\}@hw.ac.uk}}

\maketitle

\begin{abstract}
The Post Correspondence Problem is a classical decision problem about equalisers of free monoid homomorphisms.
We prove connections between several variations of this classical problem, but in the setting of free groups and free group homomorphisms.
Among other results, and working under certain injectivity assumptions, we prove that computing the rank of the equaliser of a pair of free group homomorphisms can be applied to computing a basis of this equaliser, and also to solve the ``generalised'' Post Correspondence Problem for free groups.

\keywords{Post Correspondence Problem, free group, rational constraint.}
\end{abstract}

\section{Introduction}
\label{introduction}
In this article we connect several variations of the classical Post Correspondence Problem in the setting of free groups.
The problems we consider have been open since the 1980s, and understanding how they relate and compare to their analogues in free monoids could bring us closer to their resolution.
All problems are defined in Table \ref{table:definitions}, while their status in free groups and monoids is given in Table \ref{table:results}. However, three of these problems deserve proper introductions.

The first problem we consider is the Post Correspondence Problem ($\PCP$) for free groups.
This is completely analogous to the classical Post Correspondence Problem, which is about free monoids rather than free groups, and has had numerous applications in mathematics and computer science \cite{Harju1997Morphisms}.
The $\PCP$ for other classes of groups has been successfully studied (see for example \cite[Theorem 5.8]{Myasnikov2014Post}), but it remains open for free groups, where it is defined as follows.
Let $\Sigma$ and $\Delta$ be two alphabets, let $g, h: F(\Sigma)\rightarrow F(\Delta)$ be two group homomorphisms from the free group over $\Sigma$ to the free group over $\Delta$, and
store this data in a four-tuple $I=(\Sigma, \Delta, g, h)$, called an \emph{instance} of the $\PCP$. The $\PCP$ is the decision problem:
\[
\textnormal{
Given $I=(\Sigma, \Delta, g, h)$, is there $x\in F(\Sigma)\setminus\{1\}$ such that $g(x)=h(x)$?}
\]
That is, if we consider the \emph{equaliser} $\eq(g, h)= \{x\in F(\Sigma)\mid g(x)=h(x)\}$ of $g$ and $h$, the $\PCP$ asks if $\eq(g, h)$ is non-trivial.
Determining the decidability of this problem is an important question \cite[Problem 5.1.4]{Dagstuhl2019} \cite[Section 1.4]{Myasnikov2014Post}.

Our second problem asks not just about the triviality of $\eq(g, h)$, but for a finite description of it.
We write $\PCPI$ (see Table \ref{table:definitions}) for the $\PCP$ with at least one map injective, in which case $\eq(g, h)$ is finitely generated \cite{Goldstein1986Fixed} and a finite description relates to bases:
The \emph{Basis Problem} ($\BP$) takes as input an instance $I=(\Sigma, \Delta, g, h)$ of the $\PCPI$ and outputs a basis for $\eq(g, h)$.
In Section \ref{sec:Stallings} we show that the $\BP$ is equivalent to the \emph{Rank Problem} ($\ERP$), which seeks the number of elements in the basis, and was asked by Stallings in 1984. Recent results settle the $\BP$ for certain classes of free group maps \cite{Bogopolski2016algorithm, Feighn2018algorithmic, ciobanu2020fixed, ciobanu_et_al:LIPIcs:2020:12527}, but despite this progress its solubility remains open in general. The analogous problem for free monoids, which we call the \emph{Algorithmic Equaliser Problem} ($\AEP$) (see \cite[page 2]{ciobanu_et_al:LIPIcs:2020:12527}) because it aims to describe the equaliser in terms of automata rather than bases, is insoluble \cite[Theorem 5.2]{saarela2010noneffective}.

Our third problem is the \emph{generalised} Post Correspondence Problem ($\GPCP$), which is an important generalisation of the $\PCP$ for both free groups and monoids from 1982 \cite{Ehrenfeucht1982generalized}.
For group homomorphisms $g, h: F(\Sigma)\rightarrow F(\Delta)$ and fixed elements $u_1, u_2, v_1, v_2$ of $F(\Delta)$,
an instance of the $\GPCP$ is an $8$-tuple $I_{\GPCP}=(\Sigma, \Delta, g, h, u_1, u_2, v_1, v_2)$ and the $\GPCP$ itself is the decision problem:
\begin{align*}
&\textnormal{
Given $I_{\GPCP}=(\Sigma, \Delta, g, h, u_1, u_2, v_1, v_2)$,}\\
&\textnormal{is there $x\in F(\Sigma)\setminus\{1\}$ such that $u_1g(x)u_2=v_1h(x)v_2$?}
\end{align*}

\vspace{-1cm}
\begin{table}[ht]
\caption{Summary of certain decision problems related to the $\PCP$}
\label{table:definitions}
\begin{center}
\resizebox{\columnwidth}{!}{%
\rowcolors{7}{}{lightgray}
\begin{tabular}{l|l}
 & Fixed: finite alphabets $\Sigma$ and $\Delta$ and free groups $F(\Sigma)$, $F(\Delta)$.\\
Problems & Input: homomorphisms $g, h: F(\Sigma)\rightarrow F(\Delta)$\\
(for free groups) &\\
 & Additional input for GPCP: $u_1, u_2, v_1, v_2 \in F(\Delta)$\\
 & Additional input for $\CPCP$: rational set $\mathcal{R}\subseteq F(\Sigma)$\\
 & Additional input for $\ELPCP$: $a, b \in \Sigma^{\pm 1}$, $\Omega \subset \Sigma$\\
\hline
 \hline

 & Is it decidable whether:\\
$\PCP$ & there exists $x\in F(\Sigma)\setminus\{1\}$ s.t. $g(x)=h(x)$?\\
$\GPCP$ & there exists $x\in F(\Sigma)\setminus\{1\}$ s.t. $u_1g(x)u_2=v_1h(x)v_2$?\\
$\CPCP$ & there exists $x\neq 1$ in $\mathcal{R}$ s.t. $g(x)=h(x)$? \\
$\ELPCP$ & there exists $x\in F(\Sigma)\setminus\{1\}$ s.t. $g(x)=h(x)$ and \\
& $x$ decomposes as a freely reduced word $ayb$ for some $y\in F(\Omega)$\\
\hline
$\PCPnoI$ & $\PCP$ with neither $g$, nor $h$ injective \\
$\PCPoneI$ & $\PCP$ with exactly one of $g, h$ injective \\
$\PCPtwoI$ & $\PCP$ with both $g, h$ injective\\
$\PCPI$ & $\PCPoneI \cup \PCPtwoI$ (i.e. $\PCP$ with at least one of $g, h$ injective)\\
$\PCPC$ & $\PCP$ with $g,h$ s.t. $g(y) \neq u^{-1}h(y)u$ for all $u\in F(\Delta), y \in F(\Sigma)\setminus\{1\}$\\
$\PCPCI$ & $\PCPI\cup\PCPC$\\
$\PCP(n)$& $\PCP(n)$ for alphabet size $|\Sigma|=n$\\
\hline
$\GPCP$ variants
& $\GPCPnoI$, $\GPCPI$, $\GPCPCI$, $\GPCP(n)$, etc.\\
& analogue to $\PCP$ variants\\ 
\end{tabular}%
}
\end{center}
\end{table}

For free monoids, the $\PCP$ is equivalent to the $\GPCP$ \cite[Theorem 8]{Harju1997Morphisms}. The corresponding connection for free groups is more complicated, and explaining this connection is the main motivation of this article.
In particular, the $\GPCP$ for free groups is known to be undecidable \cite[Corollary 4.2]{Myasnikov2014Post} but this proof does not imply that the $\PCP$ for free groups is undecidable (because of injectivity issues; see Section \ref{sec:NonInjective}).
In Theorem \ref{thm:mainBODY} we connect the $\PCP$ with the $\GPCP$ via a sequence of implications, and require at least one map to be injective.

\subsubsection{Main theorem}
Theorem \ref{thm:connections} summarises the connections proven in this paper (arrows are labeled by the section numbers where the implications are proven), and Section \ref{sec:MainTwo} brings all the results together.
Note that asking for both maps to be injective refines the results in this theorem, as does restricting the size of the source alphabet $\Sigma$ (see Theorem \ref{thm:connectionsBODY}).

\begin{thmletter}[Theorem \ref{thm:connectionsBODY}]
\label{thm:connections}
In finitely generated free groups the following implications hold.
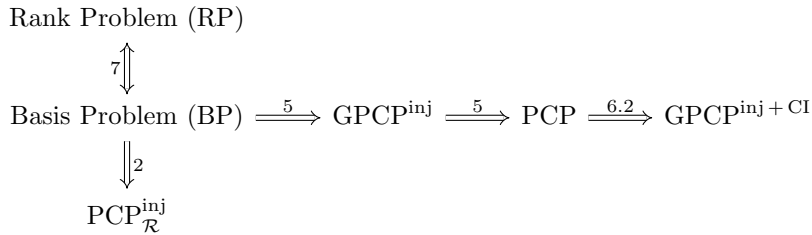
\[
\begin{tikzcd}
\operatorname{Rank \ Problem\ } (\ERP)\\
\operatorname{Basis \ Problem\ } (\BP)
\arrow[Rightarrow, "\ref{sec:MainOne}"]{r} \arrow[Rightarrow, "\ref{thm:rationalBODY}"]{d} \arrow[Leftrightarrow, "\ref{sec:Stallings}"]{u}
&
\GPCPI \arrow[Rightarrow, "\ref{sec:MainOne}"]{r}
&
\PCP\arrow[Rightarrow, "\ref{sec:conjIn}"]{r}
&
\GPCPCI\\%Note: Labels on arrows must be in *proper* quote marks, e.g. "xyz".
\CPCPI
\end{tikzcd}
\]
\end{thmletter}

This theorem implies that Stallings' Rank Problem is of central importance, as we have the chain: $\ERP \Rightarrow\GPCPI\Rightarrow\PCP$.

\subsubsection{Rational constraints}
The implication $\BP\Rightarrow\GPCPI$ above is surprising because the inputs to the problems are very different.
The proof of this implication uses the $\PCPI$ with a certain rational constraint, namely the problem $\ELPCPI$ (see Table \ref{table:definitions}). The relationship between the $\GPCP$ and the $\ELPCP$ still holds if neither map is injective.
As the $\GPCP$ for free groups is undecidable in general, this connection yields the following result.

\begin{thmletter}[Theorem \ref{thm:rationalundecidableBODY}]
\label{thm:rationalundecidable}
There exists a rational constraint $\mathcal{R}$ such that the $\CPCP$ is undecidable.
\end{thmletter}

\subsubsection{Random maps and generic behaviour.}
A different perspective on the $\PCP$ and its variations is to consider the behaviour of these problems when the pairs of homomorphisms are picked randomly (while the two alphabets $\Delta=\{x_1, \dots, x_m\}$ and $\Sigma=\{y_1, \dots, y_k\}$, and ambient free groups $F(\Delta)$ and $F(\Sigma)$ remain fixed). Any homomorphism is completely determined by how it acts on the generators, and so picking $g$ and $h$ randomly is to be interpreted as picking $(g(x_1), \dots, g(x_m))$ and $(h(x_1), \dots, h(x_m))$ as random tuples of words in $F(\Sigma)$ (see Section \ref{sec:CI} for details). There is a vast literature (see for example \cite{KMSS2003}) on the types of objects or behaviours which appear with probability $0$, called \emph{negligible}, or with probability $1$, called \emph{generic}, in infinite groups. In this spirit, the \emph{generic} $\PCP$ refers to the $\PCP$ applied to a generic set (of pairs) of maps, that is, a set of measure $1$ in the set of all (pairs of) maps, and we say that the generic $\PCP$ is decidable if the $\PCP$ is decidable for `almost all' instances, that is, for a set of measure $1$ of pairs of homomorphisms.

In Section \ref{sec:CI} we describe the setup used to count pairs of maps and compute probabilities, and show that among all pairs of maps $g,h$, the property of being \emph{conjugacy inequivalent} (that is, for every $u\in F(\Delta)$ there is no $x\neq 1$ in $F(\Sigma)$ such that $g(x)=u^{-1}h(x)u$; defined in Table \ref{table:definitions} as $\PCPC$ and $\GPCPC$) occurs with probability $1$; that is, conjugacy inequivalent maps are generic:

\begin{thmletter}[Theorem \ref{thm:genericityBODY}]
\label{thm:genericity}
Instances of the $\PCPCI$ are generic instances of the $\PCP$. That is, with probability $1$, a pair of maps is conjugacy inequivalent.
\end{thmletter}

Theorem \ref{thm:genericity} shows that the implication $\PCP\Rightarrow\GPCPCI$ in Theorem \ref{thm:connections} is the generic setting, and hence for `almost all maps' we have $\PCP\Leftrightarrow\GPCP$.

We conclude the introduction with a summary of the status the $\PCP$ and its variations for free monoids and groups. We aim to study the computational complexity of these problems and how this complexity behaves with respect to the implications proved in this paper in future work.

\vspace{-0.5cm}
\begin{table}[ht]
\caption{Status of results for free monoids and free groups}
\label{table:results}
\begin{center}
\resizebox{\columnwidth}{!}{%
\rowcolors{3}{}{lightgray}
\begin{tabular}{l | c | c | c | c }
Problems & In free monoids& References & In free groups& References \\
 & & for free monoids&& for free groups\\
\hline
 general $\PCP$ & undecidable &\cite{Post1946Correspondence} &unknown &\\
general $\AEP / \BP$ & undecidable & \cite[Theorem 5.2]{saarela2010noneffective} & unknown &  \\
$\PCPnoI$ & undecidable & \cite{Post1946Correspondence} & decidable & Lemma \ref{lem:CommonKer}\\
$\PCPI$ & undecidable & \cite{Lecerf63} & unknown &\\
$\GPCP$ &undecidable & \cite[Theorem 8]{Harju1997Morphisms} & undecidable &\cite[Corollary 4.2]{Myasnikov2014Post}\\
$\GPCPnoI$ &undecidable & \cite[Theorem 8]{Harju1997Morphisms} & undecidable & Lemma \ref{lem:PCPinj}\\
$\GPCPI$ &undecidable & \cite{Lecerf63} & unknown & \\
$\GPCPCI$ &N/A & & unknown &  \\
$\CPCP$ & undecidable & \cite{Post1946Correspondence} & undecidable &  Theorem~\ref{thm:rationalundecidable}\\
$\PCPCI$ & N/A & & unknown & \\
generic $\PCP$& decidable & \cite[Theorem 4.4]{GMMU}& decidable &  \cite{CMV}\\
\end{tabular}%
}
\end{center}
\end{table}
\section{Non-injective maps}
\label{sec:NonInjective}
In this section we investigate the $\PCPnoI$ and $\GPCPnoI$, which are the $\PCP$ and $\GPCP$ under the assumption that both maps are non-injective.

\subsubsection{The \boldmath{$\PCP$} for non-injective maps}
We first prove that the $\PCPnoI$ is trivially decidable, with the answer always being ``yes''.

\begin{lemma}
\label{lem:CommonKer}
If $g, h: F(\Sigma)\rightarrow F(\Delta)$ are both non-injective homomorphisms then $\eq(g, h)$ is non-trivial.
\end{lemma}

\begin{proof}
We prove that $\ker(g)\cap\ker(h)$ is non-trivial,
which is sufficient.
Let $c\in \ker(g)$ and $d\in \ker(h)$ be non-trivial elements. If $\langle c,d\rangle \cong \mathbb{Z}=\langle x\rangle$, there exist integers $k,l$ such that $c=x^k$ and $d=x^l$. Then $g(x^{kl})=1=h(x^{kl})$ so $x^{kl}\in\ker(g)\cap\ker(h)$ with $x^{kl}$ non-trivial, as required. If $\langle c,d\rangle \ncong \mathbb{Z}$ then $g([c,d])=1=h([c,d])$, so $[c, d]\in\ker(g)\cap\ker(h)$ with $[c, d]$ non-trivial, as required.
\qed\end{proof}

As we can algorithmically determine if a free group homomorphism is injective (e.g. via Stallings' foldings), Lemma \ref{lem:CommonKer} gives us the following:

\begin{proposition}
\label{prop:PCPnoI}
$\PCP\Leftrightarrow\PCPnoI$
\end{proposition}

\subsubsection{The \boldmath{$\GPCP$} for non-injective maps}
Myasnikov, Nikolaev and Ushakov defined the $\PCP$ and $\GPCP$ for general groups in \cite{Myasnikov2014Post}. Due to this more general setting their formulation is slightly different to ours but, from a decidability point of view, the problems are equivalent for free groups. They proved that the $\GPCP$ is undecidable for free groups; however, we now dig into their proof and observe that it assumes both maps are non-injective.
Therefore, $\GPCPI$ remains open.

\begin{lemma}
\label{lem:PCPinj}
The $\GPCPnoI$ is undecidable.
\end{lemma}

\begin{proof}
Let $H$ be a group with undecidable word problem and let $\langle \mathbf{x}\mid\mathbf{r}\rangle$ be a presentation of $H$. Set $\Delta:=\mathbf{x}$, define
\[
\Sigma:=\{(x, 1)\mid x\in\mathbf{x}\}\cup\{(1, x^{-1})\mid x\in\mathbf{x}\}\cup\{(R, 1)\mid x\in\mathbf{r}\}\cup\{(1, R^{-1})\mid R\in\mathbf{r}\}
\]
and define $g:(p, q)\mapsto p$ and $h:(p, q)\mapsto q$. Note that neither $g$ nor $h$ is injective, as if $R\in\mathbf{r}$ then $g(R, 1)$ may be realised as the image of a word over $\{(x, 1)\mid x\in\mathbf{x}\}\cup\{(1, x^{-1})\mid x\in\mathbf{x}\}$, and analogously for $h(1, R)$. Taking $w\in F(\Delta)$, the instance $(\Sigma, \Delta, g, h, w, 1, 1, 1)$ of the $\GPCPnoI$ has a solution if and only if the word $w$ defines the identity of $H$ \cite[Proof of Proposition 4.1]{Myasnikov2014Post}. As $H$ has undecidable word problem it follows that the $\GPCPnoI$ is undecidable.
\qed\end{proof}

\section{The $\GPCP$ and extreme-letter restrictions}
\label{sec:injectivity}
In this section we connect the $\GPCP$ and the $\PCP$ under a certain rational constraint.
This connection underlies Theorem \ref{thm:rationalundecidable}, as well as the implications $\BP\Rightarrow\GPCPI$ and $\PCP\Rightarrow\GPCPCI$ in Theorem \ref{thm:connections}.

Our results here, as in much of the rest of the paper, are broken down in terms of injectivity, and also alphabet sizes; understanding for which sizes of alphabet $\Sigma$ the classical Post Correspondence Problem is decidable/undecidable is an important research theme \cite{Neary2015undecidability, Ehrenfeucht1982generalized, Harju1997Morphisms}.

For an alphabet $\Sigma$, let $\Sigma^{-1}$ be the set of formal inverses of $\Sigma$, and write $\Sigma^{\pm1}=\Sigma \cup \Sigma^{-1}$. For example, if $\Sigma=\{a,b\}$ then $\Sigma^{\pm1}=\{a, b, a^{-1}, b^{-1}\}$.

The \emph{extreme-letter-restricted Post Correspondence Problem for free groups} $\ELPCP$ (see also Table \ref{table:definitions}) is the following problem: Let $g, h: F(\Sigma)\rightarrow F(\Delta)$ be two group homomorphisms, $\Omega\subset\Sigma$ a set, and $a, b\in\Sigma^{\pm1}$ two letters; an instance of the $\ELPCP$ is a $6$-tuple $I_{\ELPCP}=(\Sigma, \Delta, g, h, a, \Omega, b)$ and the $\ELPCP$ itself is:
\begin{align*}
&\textnormal{
Given $I_{\ELPCP}=(\Sigma, \Delta, g, h, a, \Omega, b)$,
is there $x\in F(\Sigma)\setminus\{1\}$ such that:}\\
&\textnormal{- $g(x)=h(x)$, and}\\
&\textnormal{- $x$ decomposes as a freely reduced word $ayb$ for some $y\in F(\Omega)$?}
\end{align*}

\subsubsection{Connecting the $\GPCP$ and the $\PCP$.}
We start with an instance $I_{\GPCP}=(\Sigma, \Delta, g, h, u_1, u_2, v_1, v_2)$ of the $\GPCP$ and consider the instance 
\[
I_{\PCP}=(\Sigma\sqcup\{B, E\}, \Delta\sqcup\{B, E, \#\}, g', h')
\]
of the $\PCP$, where $g'$ and $h'$ are defined as follows.
\begin{align*}
g'(z):=
\begin{cases}
\#^{-1}g(z)\#&\textrm{if} \ z\in\Sigma\\
B\#u_1\#&\textrm{if} \ z=B\\
\#^{-1}u_2\#E&\textrm{if} \ z=E
\end{cases}
&&
h'(z):=
\begin{cases}
\#h(z)\#^{-1}&\textrm{if} \ z\in\Sigma\\
B\#v_1\#^{-1}&\textrm{if} \ z=B\\
\#v_2\#E&\textrm{if} \ z=E
\end{cases}
\end{align*}

Injectivity is preserved through this construction since $\rk(\im(g))+2=\rk(\im(g'))$: this can be seen via Stallings' foldings, or directly by noting that the image of $g'$ restricted to $F(\Sigma)$ is isomorphic to $\im(g)$, that $B$ only occurs in $g'(d)$ and $E$ only occurs in $g'(e)$. Analogously, $h'$ is an injection if and only if $h$ is, as again $\rk(\im(h))+2=\rk(\im(h'))$. Thus we get:

\begin{lemma}
\label{lem:injectivity}
The map $g'$ is injective if and only if $g$ is, and the map $h'$ is injective if and only if $h$ is.
\end{lemma}

We now connect the solutions of $I_{\GPCP}$ to those of $I_{\PCP}$.

\begin{lemma}
\label{lem:solnforGPCP}
A word $y\in F(\Sigma)$ is a solution to $I_{\GPCP}$ if and only if the word `$ByE$' is a solution to $I_{\PCP}$.
\end{lemma}

\begin{proof}
Starting with $y$ being a solution to $I_{\GPCP}$, we obtain the following sequence of equivalent identities:
\begin{align*}
u_1g(y)u_2&=v_1h(y)v_2\\
B\#(u_1g(y)u_2)\#E&=B\#(v_1h(y)v_2)\#E\\
B\#u_1\#\cdot \#^{-1}g(y)\#\cdot \#^{-1}u_2\#E&=B\#v_1\#^{-1}\cdot \#h(y)\#^{-1}\cdot \#v_2\#E\\
g'(B)g'(y)g'(E)&=h'(B)h'(y)h'(E)\\
g'(ByE)&=h'(ByE).
\end{align*}
Therefore $ByE$ is a solution to $I_{\PCP}$, so the claimed equivalence follows.
\qed\end{proof}

Theorem \ref{thm:GPCP} includes conditions on injectivity and on alphabet sizes; see Table \ref{table:definitions} for definitions.

\begin{theorem}\leavevmode The following hold in a finitely generated free group. 
\label{thm:GPCP}
\begin{enumerate}
\item $\ELPCPoneI(n+2)\Rightarrow\GPCPoneI(n)$
\item $\ELPCPtwoI(n+2)\Rightarrow\GPCPtwoI(n)$
\end{enumerate}
\end{theorem}

\begin{proof}
Let $I_{\GPCP}$ be an instance of the $\GPCPI$, and construct from it the instance $I_{\ELPCP}=(\Sigma\sqcup\{B, E\}, \Delta\sqcup\{B, E, \#\}, g', h', B, \Sigma, E)$ of the $\ELPCP$, which is the instance $I_{\PCP}$ defined above under the constraint that solutions have the form $ByE$ for some $y\in F(\Sigma)$. 

By Lemma \ref{lem:injectivity}, $I_{\ELPCP}$ is an instance of the $\ELPCPoneI(n+2)$ if and only if $I_{\GPCP}$ is an instance of $\GPCPoneI(n)$, and similarly for $\ELPCPtwoI(n+2)$ and $\GPCPtwoI(n)$. The result then follows from Lemma \ref{lem:solnforGPCP}.
\qed\end{proof}

The above does not prove that $\PCPI\Leftrightarrow\GPCPI$, because $I_{\PCP}$ might have solutions of the form $BxB^{-1}$ or $E^{-1}xE$.
For example, if we let $I_{\GPCP}=(\{a\}, \{a, c, d\}, g, h, c, \epsilon, \epsilon, d)$ with $g(a)=a$ and $h(a)=cac^{-1}$, then there is no $x\in F(a)$ such that $cg(a)=h(a)d$, but defining $g', h'$ as above then $BaB^{-1}\in\eq(g', h')$.
In Section \ref{sec:conjIn} we consider maps where such solutions are impossible, and there the equivalence $\PCPI\Leftrightarrow\GPCPI$ does hold.

\section{The $\PCP$ under rational constraints}
For an alphabet $A$, a language $L\subseteq A^*$ is {\em regular} if there exists some finite state automaton over $A$ which accepts exactly the words in $L$. 
Let $\pi:(\Sigma^{\pm 1})^*\to F(\Sigma)$ be the natural projection map. 
A subset $R\subseteq F(\Sigma)$ is {\em rational} if $R=\pi(L)$ for some regular language $L\subseteq (\Sigma^{\pm 1})^*$. 

In this section we consider the $\CPCPI$, which is the $\PCPI$ subject to the rational constraint $\mathcal{R}$.
We prove that the $\CPCPI$ under any rational constraint $\mathcal{R}$ can be solved via the Basis Problem ($\BP$) (so $\BP\Rightarrow\CPCPI$ from Theorem \ref{thm:connections}).
We later apply this to prove $\BP\Rightarrow\GPCPI$ from Theorem \ref{thm:connections}, as the $\ELPCPI$ from Section \ref{sec:injectivity} is simply the $\PCPI$ under a specific rational constraint.

\begin{theorem}\leavevmode
The following hold in a finitely generated free group. 
\label{thm:rationalBODY}
\begin{enumerate}
\item $\AEPoneIb(n)\Rightarrow\CPCPoneI(n)$
\item $\AEPtwoIb(n)\Rightarrow\CPCPtwoI(n)$
\end{enumerate}
\end{theorem}
\begin{proof}
Let $g, h$ be homomorphisms from $F(\Sigma)$ to $F(\Delta)$ such that at least one of them is injective. Their equaliser $\eq(g,h)$ is a finitely generated subgroup of $F(\Sigma)$ \cite{Goldstein1986Fixed}, so $\eq(g,h)$ is a rational set (see for example Section 3.1 in \cite{bartholdi2010rational}).

As the Basis Problem is soluble, we can compute a basis for $\eq(g, h)$. This is equivalent to finding a finite state automaton $\mathcal{A}$ (called a ``core graph'' in the literature on free groups; see \cite{Kapovich2002Stallings}) that accepts the set $\eq(g,h)$.

Let $\mathcal{R}$ be a rational set in $F(\Sigma)$. The $\CPCP$ for $g$ and $h$ is equivalent to determining if there exists any non-trivial $x \in \mathcal{R} \cap \eq(g,h)$. Since the intersection of two rational sets is rational, and an automaton recognising this intersection is computable by the standard product construction of automata, one can determine whether $\mathcal{R} \cap \eq(g,h)$ is trivial or not, and thus solve $\CPCP$.
\qed\end{proof}

The following theorem immediately implies Theorem \ref{thm:rationalundecidable}, as restrictions on the maps do not affect the rational constraints.

\begin{theorem}[Theorem \ref{thm:rationalundecidable}]
\label{thm:rationalundecidableBODY}
$\ELPCPnoI$ is undecidable in free groups.
\end{theorem}

\begin{proof}
We have $\ELPCPnoI\Rightarrow\GPCPnoI$ by Lemmas \ref{lem:injectivity} and \ref{lem:solnforGPCP}.
The result follows as $\GPCPnoI$ is undecidable by Lemma \ref{lem:PCPinj}.
\qed\end{proof}

\section{Main results, part 1}
\label{sec:MainOne}
Here we combine results from the previous sections to prove certain of the implications in Theorem \ref{thm:connections}.
The implications we prove refine Theorem \ref{thm:connections}, as they additionally contain information on alphabet sizes and on injectivity.
\begin{theorem}\leavevmode The following hold in finitely generated free groups.
\label{thm:mainBODY}
\begin{enumerate}
\item\label{mainBODY:1} $\AEPoneIb(n+2)\Rightarrow\GPCPoneI(n)\Rightarrow\PCPoneI(n)$
\item\label{mainBODY:2} $\AEPtwoIb(n+2)\Rightarrow\GPCPtwoI(n)\Rightarrow\PCPtwoI(n)$
\end{enumerate}
\end{theorem} 

\begin{proof}
As $\ELPCPoneI$ is an instance of the $\PCPI$ under a rational constraint, Theorem \ref{thm:rationalBODY} gives us that $\AEPoneIb(n+2)\Rightarrow\ELPCPoneI(n+2)$, while Theorem \ref{thm:GPCP} gives us that $\ELPCPoneI(n+2)\Rightarrow\GPCPoneI(n)$, and the implication $\GPCPoneI(n)\Rightarrow\PCPoneI(n)$ is obvious as instances of the $\PCP$ are instances of the $\GPCP$ but with empty constants $u_i, v_i$. Sequence (\ref{mainBODY:1}), with one map injective, therefore holds, while the proof of sequence (\ref{mainBODY:2}) is identical.
\qed\end{proof}

Removing the injectivity assumptions gives the following corollary; the implications $\BP\Rightarrow\GPCPI\Rightarrow\PCP$ of Theorem \ref{thm:connections} follow immediately.

\begin{corollary}
\label{corol:mainBODY}
$\BP(n+2)\Rightarrow\GPCPI(n)\Rightarrow\PCP(n)$
\end{corollary} 

\begin{proof}
Theorem \ref{thm:mainBODY} gives that $\BP(n+2)\Rightarrow\GPCPI(n)\Rightarrow\PCPI(n)$, while $\PCP(n)\Leftrightarrow\PCPI(n)$ by Proposition \ref{prop:PCPnoI}.
\qed\end{proof}

\section{Conjugacy inequivalent maps}
\label{sec:CI}
In this section we prove genericity results and give conditions under which the $\PCP$ implies the $\GPCP$.
In particular, we prove Theorem \ref{thm:genericity}, and we prove the implication $\PCP\Rightarrow\GPCPCI$ from Theorem \ref{thm:connections}.

A pair of maps $g, h: F(\Sigma)\rightarrow F(\Delta)$ is said to be \emph{conjugacy inequivalent} if for every $u\in F(\Delta)$ there does not exist any non-trivial $x\in F(\Sigma)$ such that $g(x)=u^{-1}h(x)u$ (see Table \ref{table:definitions}). For example, if the images of $g, h: F(\Sigma)\rightarrow F(\Delta)$ are \emph{conjugacy separated}, that is, if $\im(g)\cap u^{-1}\im(h)u$ is trivial for all $u\in F(\Delta)$, then $g$ and $h$ are conjugacy inequivalent. 
We write $\PCPCI$/$\GPCPCI$ for those instances of the $\GPCPI$/$\PCPI$ where the maps are conjugacy inequivalent.

\subsection{Random maps and genericity}

Here we show that among all pairs of maps $g, h: F(\Sigma)\rightarrow F(\Delta)$, the property of being conjugacy inequivalent occurs with probability $1$; that is, conjugacy inequivalent maps are \emph{generic}.

\begin{theorem}[Theorem \ref{thm:genericity}]
\label{thm:genericityBODY}
Instances of the $\PCPCI$ are generic instances of the $\PCP$. That is, with probability $1$, a pair of maps is conjugacy inequivalent.
\end{theorem}

Before we prove the theorem, we need to describe the way in which probabilities are computed. We fix the two alphabets $\Delta=\{x_1, \dots, x_m\}$ and $\Sigma=\{y_1, \dots, y_k\}$, $m \leq k$, and ambient free groups $F(\Delta)$ and $F(\Sigma)$, and pick $g$ and $h$ randomly by choosing $(g(x_1), \dots, g(x_m))$ and $(h(x_1), \dots, h(x_m))$ independently at random, as tuples of words of length bounded by $n$ in $F(\Sigma)$. If $\mathcal{P}$ is a property of tuples (or subgroups) of $F(\Sigma)$, we say that \emph{generically many} tuples (or finitely generated subgroups) of $F(\Sigma)$ satisfy $\mathcal{P}$ if the proportion of $m$-tuples of words of length $\leq n$ in $F(\Sigma)$ which satisfy $\mathcal{P}$ (or generate a subgroup satisfying $\mathcal{P}$), among all possible $m$-tuples of words of length $\leq n$, tends to 1 when $n$ tends to infinity.

\begin{proof}
%[Theorem \ref{thm:genericityBODY}]
Let $n >0$ be an integer, and let $(a_1, \dots, a_m)$ and $(b_1, \dots, b_m)$ be two tuples of words in $F(\Sigma)$ satisfying length inequalities $|a_i|\leq n$ and $|b_i| \leq n$ for all $i$. We let the maps $g, h: F(\Sigma)\rightarrow F(\Delta)$ that are part of an instance of $\PCP$ be defined as $g(x_i)=a_i$ and $h(x_i)=b_i$, and note that the images $\im(g)$ and $\im(h)$ in $F(\Delta)$ are subgroups generated by $(a_1, \dots, a_m)$ and $(b_1, \dots, b_m)$, respectively. 

We claim that among all $2m$-tuples $(a_1, \dots, a_m, b_1, \dots, b_m)$ with $|a_i|, |b_i| \leq n$, a proportion of them tending to $1$ as $n \rightarrow \infty$ satisfy (1) the subgroups $L=\langle a_1, \dots, a_m\rangle$ and $K=\langle b_1, \dots, b_m \rangle$ are both of rank $m$, and (2) for every $u \in F(\Delta)$ we have $L^u \cap K=\{1\}$. Claim (1) is equivalent to $g,h$ being generically injective, and follows from \cite{MTV}, while claim (2) is equivalent to  $\im(g)^u \cap \im(h)=\{1\}$ for every $u \in F(\Delta)$, which implies $g$ and $h$ are generically conjugacy separated, and follows from \cite[Theorem 1]{CMV}. More specifically, \cite[Theorem 1]{CMV} proves that for any tuple $(a_1, \dots, a_m)$, `almost all' (precisely computed) tuples $(b_1, \dots, b_m)$, with $|b_i| \leq n$, give subgroups $L=\langle a_1, \dots, a_m\rangle$ and $K=\langle b_1, \dots, b_m \rangle$ with trivial pullback, that is, for every $u \in F(\Delta)$, $K^u \cap L=\{1\}$. Going over all $(a_1, \dots, a_m)$ with $|a_i| \leq n$ and counting the tuples $(b_1, \dots, b_m)$ (as in \cite{CMV}) satisfying property (2) gives the genericity result for all $2m$-tuples.
\qed\end{proof}

\subsection{The $\GPCP$ for conjugacy inequivalent maps}
\label{sec:conjIn}
We now prove that the $\PCP$ implies the $\GPCPCI$ and hence that, generically, the $\PCP$ implies the $\GPCP$.
Recall that if $I_{\GPCP}$ is a specific instance of the $\GPCP$ we can associate to it a specific instance $I_{\PCP}=(\Sigma\sqcup\{B, E\},\Delta\sqcup\{B, E, \#\}, g', h')$, as in Section \ref{sec:injectivity}. We start by classifying the solutions to $I_{\PCP}$.

\begin{lemma}
\label{lem:minimalsolns}
Let $I_{\GPCP}$ be an instance of the $\GPCPI$, with associated instance $I_{\PCP}$ of the $\PCPI$.
Every solution to $I_{\PCP}$ is a product of solutions of the form $(BxE)^{\pm1}$, $E^{-1}xE$ and $BxB^{-1}$, for $x\in F(\Sigma)$.
\end{lemma}

Lemma \ref{lem:minimalsolns} is proven in the appendix.
We now have:

\begin{theorem}
\label{thm:PCPvsGPCP2}
Let $I_{\GPCP}=(\Sigma, \Delta, g, h, u_1, u_2, v_1, v_2)$ be an instance of the $\GPCPI$, such that there is no non-trivial $x\in F(\Sigma)$ with $u_1g(x)u_1^{-1}=v_1h(x)v_1^{-1}$ or $u_2^{-1}g(x)u_2=v_2^{-1}h(x)v_2$. Then $I_{\GPCP}$ has a solution (possibly trivial) if and only if the associated instance $I_{\PCP}$ of the $\PCPI$ has a non-trivial solution.
\end{theorem}

\begin{proof}
By Lemma \ref{lem:solnforGPCP}, if $I_{\GPCP}$ has a solution then $I_{\PCP}$ has a non-trivial solution. For the other direction, note that the assumptions in the theorem are equivalent to $I_{\GPCP}$ having no solutions of the form $BxB^{-1}$ or $E^{-1}xE$, and so by Lemma \ref{lem:minimalsolns}, every non-trivial solution to $I_{\GPCP}$ has the form $Bx_1E\cdots Bx_nE$ for some $x_i\in F(\Sigma)$.
The $Bx_iE$ subwords block this word off into chunks, and we see that each such word is a solution to $I_{\PCP}$. By Lemma \ref{lem:solnforGPCP}, each $x_i$ is a solution to $I_{\GPCP}$.
Hence, if $I_{\PCP}$ has a non-trivial solution then $I_{\GPCP}$ has a solution.
\qed\end{proof}

Theorem \ref{thm:PCPvsGPCP2} depends both on the maps $g$ and $h$ and on the constants $u_i$, $v_i$. The definition of conjugacy inequivalent maps implies that the conditions of Theorem \ref{thm:PCPvsGPCP2} hold always, independent of the $u_i$, $v_i$. We therefore have:

\begin{theorem}\leavevmode
The following hold in finitely generated free groups.
\label{thm:ConjInequivalentBODY}
\begin{enumerate}
\item $\PCPoneI(n+2)\Rightarrow\GPCPConeI(n)$
\item $\PCPtwoI(n+2)\Rightarrow\GPCPCtwoI(n)$
\end{enumerate}
\end{theorem}

Removing the injectivity assumptions gives the following corollary; the implication $\PCP\Rightarrow\GPCPCI$ of Theorem \ref{thm:connections} follow immediately.

\begin{corollary}
\label{corol:ConjInequivalentBODY}
$\PCP(n+2)\Rightarrow\GPCPCI(n)$
\end{corollary} 

\begin{proof}
Theorem \ref{thm:ConjInequivalentBODY} gives us that $\PCPI(n+2)\Rightarrow\GPCPCI(n)$, while the $\PCP(n)$ and $\PCPI(n)$ are equivalent by Proposition \ref{prop:PCPnoI}.
\qed\end{proof}

\section{The Basis Problem and Stallings' Rank Problem}
\label{sec:Stallings}
In this section we link the Basis Problem to Stallings' Rank Problem.
Clearly the Basis Problem solves the Rank Problem, as the rank is simply the size of the basis.
We prove that these problems are equivalent, with Lemma \ref{lem:Stallings} providing the non-obvious direction of the equivalence.
Combining this equivalence with Corollary \ref{corol:mainBODY} gives: $\ERP \Rightarrow\GPCPI\Rightarrow\PCP$.

The proof of Lemma \ref{lem:Stallings} is based on the following construction of Goldstein--Turner, which they used to prove that $\eq(g, h)$ is finitely generated \cite{Goldstein1986Fixed}.
Let $F=F(\Delta)$ be a finitely generated free group, let $H$ be a subgroup with basis $\mathcal{N}=\{\alpha_1, \ldots, \alpha_k\}$ of $F$, where $\alpha_i \in F$, and let $\phi: H\rightarrow F$ be a homomorphism. The \emph{derived graph $D_{\phi}$ of $\phi$ relative to $\mathcal{N}$} is the graph with vertex set $F$, and directed edges labeled by elements of $\mathcal{N}$ according to the rule:
if $u=\phi(\alpha_i)v\alpha_i^{-1}$ then connect $u$ to $v$ by an edge with initial vertex $u$, end vertex $v$, and label $\alpha_i$.
We shall use the fact that a word $w(\alpha_1, \ldots, \alpha_k)$ is fixed by $\phi$ if and only if the path which starts at the vertex ${\epsilon}$ (the identity of $F$) and is labeled by  $w(\alpha_1, \ldots, \alpha_k)$ is a loop, so also ends at ${\epsilon}$.

\begin{lemma}
\label{lem:Stallings}
There exists an algorithm with input an instance $I=(\Sigma, \Delta, g, h)$ of the $\PCPI$ and the rank $\rk(\eq(g, h))$ of the equaliser of $g$ and $h$, and output a basis for $\eq(g, h)$.
\end{lemma}

The following shows that Stallings' Rank Problem is equivalent to the $\BP$.

\begin{theorem}\leavevmode
The following hold in finitely generated free groups.
\label{thm:Stallings}
\begin{enumerate}
\item $\BPoneI(n+2)\Leftrightarrow\RPoneI(n)$
\item $\BPtwoI(n+2)\Leftrightarrow\RPtwoI(n)$
\end{enumerate}
\end{theorem}

\begin{proof}
Let $I_{\PCP}$ be an instance of the $\PCPI$.
As the rank of a free group is precisely the size of some (hence any) basis for it, if we can compute a basis for $\eq(g, h)$ then we can compute the rank of $\eq(g, h)$.
On the other hand, by Lemma \ref{lem:Stallings} if we can compute the rank of $\eq(g,h)$ then we can compute a basis of $\eq(g, h)$.
\qed\end{proof}

We therefore have:

\begin{corollary}
\label{corol:Stallings}
$\BP(n)\Leftrightarrow\ERP(n)$
\end{corollary}

\section{Main results, part 2}
\label{sec:MainTwo}
We now combine results from the previous sections to the following result, from which Theorem \ref{thm:connections} follows immediately.

\begin{theorem}
\label{thm:connectionsBODY}
In finitely generated free groups the following implications hold.
\[
\begin{tikzcd}
\ERP(n)\\
\BP(n) \arrow[Rightarrow]{r} \arrow[Rightarrow]{d} \arrow[Leftrightarrow]{u}
&
\GPCPI(n) \arrow[Rightarrow]{r}
&
\PCP(n-2)\arrow[Rightarrow]{r}
&
\GPCPCI(n)\\%Note: Labels on arrows must be in *proper* quote marks, e.g. "xyz".
\CPCPI(n)
\end{tikzcd}
\]
\end{theorem}
\begin{proof} The proof is a summary of the results already established in the rest of the paper, and we give a schematic version of it here.

$\ERP(n)\Leftrightarrow\BP(n)$ holds by Corollary \ref{corol:Stallings}.

$\BP(n)\Rightarrow\CPCPI(n)$ holds by Theorem \ref{thm:rationalBODY}.

$\BP(n) \Rightarrow\GPCPI(n)\Rightarrow\PCP(n)$ holds by Corollary \ref{corol:mainBODY}.

$\PCP(n-2)\Rightarrow\GPCPCI(n)$ holds by Corollary \ref{corol:ConjInequivalentBODY}.
\qed\end{proof}

\bibliographystyle{splncs04}
\bibliography{BibTexBibliography}

\newpage

\section*{Appendix: Additional material to Sections \ref{sec:CI} \& \ref{sec:Stallings}}

\begin{proof}[Lemma \ref{lem:minimalsolns}]
Let $x'$ be a solution to $I_{\PCP}$, and decompose it as a freely reduced word
\[
x_0\alpha_1x_1\alpha_2\cdots x_{n-1}\alpha_{n}x_n
\]
for $x_i\in F(\Sigma)$ and $\alpha_i\in\{B, E\}^{\pm1}$.

We shall refer to the letters $B^{\pm1}, E^{\pm1}$ as \emph{markers}. Indeed, they act as ``separators'' in the word $x'$, because in the definitions of $g'$ and $h'$ the letter $B$ only occurs in $g'(B)$ and $h'(B)$, and the letter $E$ only occurs in $g'(E)$ and $h'(E)$; thus either the $B^{\pm1}$ and $E^{\pm1}$ terms in $x'$ are preserved under application of $g'$ and $h'$, or two of these letters cancel in the image.
Therefore, the parts of each of $g'(\alpha_{i}x_i\alpha_{i+1})$ and $h'(\alpha_{i}x_i\alpha_{i+1})$ lying between the markers must be equal, or otherwise there exist some indices $i, j$ such that both $g'(\alpha_{i}x_i\alpha_{i+1})$ and $h'(\alpha_{j}x_j\alpha_{j+1})$ are contained in $F(\Delta)$: this collapsing must happen in each image as the numbers of $B^{\pm1}$ and $E^{\pm1}$ terms must be equal in both images.
By inspecting the maps, we see that this collapsing requires $x_{i}\in\ker(g')$ and $\alpha_{i}=\alpha_{i+1}^{-1}$, and $x_{j}\in\ker(h')$ and $\alpha_{j}=\alpha_{j+1}^{-1}$.
However, by Lemma \ref{lem:injectivity} only one of $g'$ or $h'$ is injective, $g'$ say, so then $x_{i}$ is trivial, and so $\alpha_{i}x_{i}\alpha_{i+1}$ is empty, contradicting the decomposition being freely reduced.
Therefore, we have that for all $1\leq i< n$ the parts of each of $g'(\alpha_{i}x_i\alpha_{i+1})$ and $h'(\alpha_{i}x_i\alpha_{i+1})$ lying between the markers $B^{\pm1}, E^{\pm1}$ are equal.
This also implies that the parts of $g'(x_0\alpha_1)$ and $h'(x_0\alpha_1)$ lying before the markers are equal.

We next prove that $x_0$ is empty, $\alpha_1\in\{B, E^{-1}\}$, $\alpha_2\in\{B^{-1}, E\}$ (so in particular, non-empty), and that $\alpha_1x_1\alpha_2$ is a solution to $I_{\PCP}$.

Suppose $x_0$ is non-empty. Then either $g'(x_0\alpha_1)$ starts with $\#^{-1}$ or $E$ (when $u_2$ is empty), or $x_0\in\ker g'$ and the image starts with $B$ or $E^{-1}$, while either $h'(x_0\alpha_1)$ starts with $\#$ or $B^{-1}$ (when $v_1$ is empty), or $x_0\in\ker h'$ and the image starts with $B$ or $E^{-1}$.
As one map is injective, by Lemma \ref{lem:injectivity}, we have that $g'(x_0\alpha_1)$ and $h'(x_0\alpha_1)$ start with different letters, contradicting the above paragraph.
Hence, $x_0$ is empty.
We then have that $\alpha_1\in\{B, E^{-1}\}$, as this is the only way that the parts of $g'(\alpha_1)$ and $h'(\alpha_1)$ lying before the marker can be equal.

Suppose $\alpha_2$ is empty, so $x'=\alpha_1x_1$. Then as $x'^{-1}$ is also a solution to $I_{\PCP}$, we also have that $x_1$ is empty, and so $x'=\alpha_1\in\{B, E^{-1}\}$. However, neither $B$ nor $E^{-1}$ is a solution to $I_{\PCP}$, a contradiction.
Hence, $\alpha_2$ is non-empty.
Moreover, $\alpha_2\in\{B^{-1}, E\}$, as otherwise the parts of each of $g'(\alpha_{1}x_1\alpha_2)$ and $h'(\alpha_{1}x_1\alpha_2)$ lying between the markers are non-equal (for example, $g'(Bx_1B)=B\#u_1g(x_1)B\cdots$ while $h'(Bx_1B)=B\#v_1h(x_1)\#^{-1}B\cdots$).
Finally, under these restrictions on $\alpha_{1}$ and $\alpha_{2}$, and because the parts of $g'(\alpha_{1}x_1\alpha_2)$ and $h'(\alpha_{1}x_1\alpha_2)$ lying between the markers are equal, we get that $\alpha_{1}x_1\alpha_2$ is a solution to $I_{\PCP}$ of the form $(Bx_1E)^{\pm1}$, $E^{-1}x_1E$, or $Bx_1B^{-1}$ for $x_1\in F(\Sigma)$.

We can now prove the result: the product $(\alpha_1x_1\alpha_2)^{-1}x'$ is also a solution to $I_{\PCP}$, and it decomposes as $x_2\alpha_3\cdots x_{n-1}\alpha_{n}x_n$. Applying the above argument, we get that $x_2$ is empty and that $\alpha_{3}x_3\alpha_4$ is a solution to $I_{\PCP}$ of the form $(Bx_3E)^{\pm1}$, $E^{-1}x_3E$, or $Bx_3B^{-1}$ for $x_3\in F(\Sigma)$. Repeatedly reducing the solution like this, we see that $x'$ decomposes as $x''x_n$ where $x''$ is a product of solutions of the form $(BxE)^{\pm1}$, $E^{-1}xE$, and $BxB^{-1}$ for $x\in F(\Sigma)$. The result then follows as $x_n$ is empty, which can be seen by applying the above argument to the solution $x'^{-1}$, which decomposes as $x_n\alpha_{n-1}\cdots$.
\qed\end{proof}

\begin{proof}[Lemma \ref{lem:Stallings}]
Suppose without loss of generality that $g$ is injective, and consider the homomorphism $\phi=h\circ g^{-1}: \im(g)\rightarrow F(\Delta)$.
Then $g(\eq(g, h))$ is precisely the fixed subgroup $\fix(\phi)=\{x\in \im(g)\mid \phi(x)=x\}$, and as $g$ is injective it restricts to an isomorphism between these two subgroups.
We therefore give an algorithm which takes as input $I=(\Sigma, \Delta, g, h)$ and the rank of $\eq(g, h)$ and outputs a basis $\mathcal{B}$ for $\fix(\phi)$; this is sufficient as $g^{-1}(\mathcal{B})$ is then an algorithmically computable basis for $\eq(g, h)$.
If $\rk(\eq(g, h))=0$ then the basis is the empty set, so we may assume $\rk(\eq(g, h))\geq1$ (this reduction is not necessary, but not doing so introduces certain subtleties).

Let $\Gamma_{\phi}$ be the the union of those loops in $D_{\phi}$ which contain no degree-$1$ vertices, but contain the vertex ${\epsilon}$; this is simply the component of the core graph of $D_{\phi}$ which contains the specified vertex.
As labels of loops in $D_{\phi}$ correspond to elements of $\fix(\phi)$, any basis of the fundamental group $\pi_1(\Gamma_{\phi})$ corresponds to a basis of $\fix(\phi)$.
Now, a basis for $\pi_1(\Gamma_{\phi})$ can be computed via standard algorithms (see for example \cite[Propositions 6.7]{Kapovich2002Stallings}), and therefore to prove the result we only need to construct the graph $\Gamma_{\phi}$.

Start with vertex set $V$ consisting of a single vertex ${\epsilon}$ corresponding to the empty word.
Now enter a loop, terminating when $|E|-|V|=\rk(\eq(g, h))-1$, as follows:
For each $v\in V$ and each $i\in \{1, \ldots, k\}$, add an edge to $E$ starting at $v$ and ending at the vertex corresponding to the element $\phi(\alpha_i)g\alpha_i^{-1}$ of $F(\Sigma)$; if there is no such vertex in $V$ then first add one to $V$.
When the loop terminates, prune the resulting graph to obtain a graph $\Gamma_{\phi}'$ by iteratively removing all degree-$1$ vertices.

Note that $\rk(\eq(g, h))$ is known, so we can determine if a graph satisfies $|E|-|V|=\rk(\eq(g, h))-1$; in particular, at each iteration in the looping procedure we know whether to continue or to terminate the loop.

There are two things to prove.
Firstly, that the looping procedure terminates (and so the above is actually an algorithm), and secondly that the terminating graph $\Gamma_{\phi}'$ is in fact the graph $\Gamma_{\phi}$.
So, note that the procedure constructs \emph{some} subgraph of $D_{\phi}$ which contains the empty word as a vertex, which has no degree-$1$ vertices, and which satisfies $\#\text{edges}-\#\text{vertices}= \rk(\eq(g, h))-1$, or no such subgraph exists.
Now, the graph $\Gamma_{\phi}$ satisfies these conditions (because its fundamental group has rank $\rk(\eq(g, h))$), so the procedure does terminate.
Moreover, every loop in $D_{\phi}$ which contains the empty word is a loop in $\Gamma_{\phi}$, and so the subgraph constructed is in fact a subgraph of $\Gamma_{\phi}$.
The result then follows because no proper subgraph of $\Gamma_{\phi}$ satisfies $\#\text{edges}-\#\text{vertices}= \rk(\eq(g, h))-1$ (subgraphs in fact satisfy $\#\text{edges}-\#\text{vertices}\lneq \rk(\eq(g, h))-1$).
\qed\end{proof}
\end{document}